\documentclass[letterpaper, 10 pt, conference, draftcls, onecolumn]{ieeeconf} 
\IEEEoverridecommandlockouts 
\usepackage{fancyhdr}
\usepackage[margin=1in]{geometry}
\usepackage{amsmath}
\usepackage{graphicx}
\usepackage{amsfonts}
\usepackage{amssymb}
\usepackage{fixltx2e}
\usepackage{enumerate}
\usepackage{changepage}
\usepackage{mathtools}
\usepackage{color}
\usepackage{epstopdf}
\usepackage{empheq}
\usepackage{algorithm}
\usepackage{algpseudocode}

\usepackage{caption}
\usepackage{subcaption}

\usepackage{tikz}
\usetikzlibrary{matrix,positioning}

\mathtoolsset{showonlyrefs=true}

\newtheorem{definition}{Definition}

\newtheorem{lemma}{Lemma}

\newtheorem{theorem}{Theorem}
\newtheorem{talgorithm}{Algorithm}
\newtheorem{assumption}{Assumption}

\newtheorem{problem}{Problem}

\author{Stefan Hochhaus$^{\star,\dagger}$ and Matthew T. Hale$^{\star}$\thanks{$^{\star}$The authors are
    with the Department of Mechanical and Aerospace Engineering,
    University of Florida, Gainesville, FL, USA. Email:
    \texttt{smhochhaus@ufl.edu, matthewhale@ufl.edu}.}
    \thanks{$^\dagger$ Corresponding author.}
}

\begin{document}
\title{Asynchronous Distributed Optimization with \\ Heterogeneous
  Regularizations and Normalizations}
\maketitle
 
\begin{abstract}
As multi-agent networks grow in size and scale, they become increasingly
difficult to synchronize, though agents must work together even when generating and sharing different
information at different times. Targeting such cases, this paper presents an asynchronous
optimization framework in which the time between
successive communications and computations is unknown and unspecified for each agent.
Agents' updates are carried out in blocks, with each agent updating only
a small subset of all decision variables.
To provide robustness to asynchrony, each agent uses an independently chosen Tikhonov regularization. Convergence
is measured with respect to a weighted block-maximum norm in which
convergence of agents' blocks can be measured in different p-norms and weighted differently
to heterogeneously normalize problems. Asymptotic convergence is shown and convergence rates are derived explicitly in
terms of a problem's parameters, with only mild restrictions imposed upon them.
Simulation results are provided to verify the theoretical
developments made.
\end{abstract}

\section{Introduction}

Distributed optimization techniques have been applied in many areas
ranging from sensor networks \cite{Khan2009,Cortes2004,Rabbat2004} and communications \cite{Mitra1994,Chiang2007}, to robotics \cite{Soltero2014} and smart power
grids \cite{Caron2010}.
With this diversity in applications, there have emerged correspondingly
diverse problem formulations which address a wide variety of practical
considerations.
As multi-agent systems become increasingly complex, a key practical
consideration is the ability to tightly couple agents and the
timing of their behaviors. Often, perfect synchrony among agents'
communications and computations is difficult or impossible because closely coupling all agents in large
networks is also difficult or impossible.
Instead, one must sometimes utilize information that is asynchronously
generated and shared. This paper examines how to do so
in a distributed optimization setting.

There is a significant existing literature on distributed
optimization, including a large corpus of work on asynchronous optimization. One common approach is to assume that delays in communications
and computations are bounded, and this approach is used for example in
 \cite{Chen2012,Jadbabaie2003,Moreau2005,Nedic2007,Nedic2009,Nedic2010,
 Olshevsky2009,Ren2005,Touri2009},
and the delay bound parameter explicitly appears in convergence
rates in \cite{Chen2012,Nedic2007,Nedic2009,Olshevsky2009,Touri2009}. However, in some cases, delay bounds cannot be enforced. For example,
agents with mutually interfering communications may be unable to ensure that delay lengths stay below a certain threshold because delays are outside their control. Similarly, agents facing anti-access/area-denial
(A2AD) measures may be unable to predict when transmissions will be
received or even measure delay lengths at all. As a result, some works have addressed asynchronous optimization
with unbounded delays. Early work in this area includes \cite{Bertsekas1989},
as well as \cite{bertsekas1989parallell}, which gives a textbook-level treatment
and simplified proof of the main results in~\cite{Bertsekas1989}.

Work in \cite{Bertsekas1989} was expanded upon in \cite{Hale2017}, where
it was shown that a fixed Tikhonov regularization implies the existence
of the nested sets required in \cite{Bertsekas1989} for asymptotic convergence.
However, developments in \cite{Hale2017} require every agent to apply the same
regularization, which can be difficult to enforce and verify in practice,
especially in large decentralized networks.
Moreover, convergence in \cite{Hale2017} is measured with respect to
the same un-weighted norm for all agents. There is a wide variety of statistical
and machine learning problems which must be normalized due to disparate numerical scales across potentially
many orders of magnitude \cite{bishop1995neural}, and which may require measuring convergence of different components in different norms. While such problems are commonly solved using distributed optimization techniques, they are not accounted
for by the work in \cite{Hale2017}. Therefore, a fundamentally new approach is required
to account for heterogeneous regularizations and normalizations in the setting of distributed optimization.

In this paper we develop an asynchronous optimization framework to
address this gap. In particular, we examine set-constrained optimization
problems with potentially non-separable cost functions, and we allow agents' communications and computations to be arbitrarily
asynchronous, subject only to mild assumptions. Agents are permitted
to independently choose regularization parameters with no restrictions on
the disparity between them. Under these conditions, agents' convergence is measured with respect
to a weighted block-maximum norm which allows for heterogeneous normalizations
of agents' distance to an optimum in order to accommodate problems with different
numerical scales. Convergence rates are developed
in terms of agents' communications and computations without specifying
when they must occur. The framework developed in this paper uses a block-based update scheme in which each agent updates only a subset of all decision variables in a problem in order to provide a scalable update law for large convex programs. The contributions of this paper therefore consist of a scalable optimization framework that accommodates heterogeneous regularizations and normalizations, together with its convergence rate.


The rest of the paper is organized as follows. Section~II defines
the optimization problems to be solved and regularizations used. Next,
Section~III defines the block-based multi-agent update law, and Section~IV
proves its convergence and derives its convergence rate. After that, Section~V
presents simulation results and Section~VI provides concluding remarks.

\section{Tikhonov Regularization and Problem Statement}
In this section we describe the class of problems to be solved and the
assumptions imposed upon problem data. We then introduce heterogeneous
regularizations and the need for heterogeneous normalizations. Then we give
a formal problem statement that is the focus of the remainder of the paper.

We consider convex optimization problems spread across teams
of agents. In particular, we consider teams comprised of $N$ agents, where
agents are indexed over
$i\in\left[N\right]\coloneqq\left\{ 1,\ldots,N\right\}$. Agent $i$ has a
decision variable $x_{i}\in\mathbb{R}^{n_{i}}$, $n_{i}\in\mathbb{N}$, which
we refer to as its state, and we allow for $n_{i}\neq n_{j}$ when $i\neq j$.
The state $x_{i}$ is subject to the set constraint
$x_{i}\in X_{i}\subset\mathbb{R}^{n_{i}}$,
which can represent, e.g., that a mobile robot must stay in a given area.
We make the following assumption about each $X_{i}$.

\begin{assumption}
For all $i\in\left[N\right]$, the set $X_{i}\subset\mathbb{R}^{n_{i}}$ is
non-empty, compact, and convex. \hfill$\triangle$
\end{assumption}

Towards making a formal problem statement,
we aggregate agents' set constraints by defining $X\coloneqq X_1\times\dots\times X_N$, and Assumption 1 ensures that $X$ is also non-empty, compact,
and convex. We further define the ensemble state as $x\coloneqq\left(x_{1}^{T},\ldots,x_{N}^{T}\right)^{T}\in X\subset\mathbb{R}^{n}$,
where $n=\underset{i\in\left[N\right]}{\sum}n_{i}$. We consider problems in which each agent has a local
objective function $f_{i}$ to minimize, which can represent, e.g., a mobile
robot's desire to minimize its distance to a target location; only agent $i$ needs to know $f_{i}$. The agents are
also collectively subject to a coupling cost $c$, which can represent the
cost of communication congestion in a network, and we allow for $c$ to be non-separable.
We then make the following assumption about the functions $f_{i}$ and $c$.
\begin{assumption}
The functions $f_{i}$, $i\in\left[N\right]$, and $c$ are convex and $C^{2}$
(twice continuously differentiable) in $x_{i}$ and $x$, respectively. \hfill$\triangle$
\end{assumption}

In particular, $\nabla{f}$ is Lipschitz and we denote its Lipschitz constant by $L$. The sum of these costs then gives the aggregate cost function
\begin{equation}
f\left(x\right)\coloneqq c\left(x\right)+\underset{i\in[N]}{\sum}f_{i}\left(x_{i}\right),
\end{equation}
and the agents will jointly minimize $f$. For simplicity of the forthcoming analysis, we assume that $f$ has a unique minimizer. To endow $f$ with an inherent
robustness to asynchrony, we will regularize it before agents start optimizing. In
particular, we regularize $f$ on a per-agent basis, where agent $i$ uses the
regularization parameter $\alpha_{i}>0$ and where we allow $\alpha_{i}\neq\alpha_{j}$ when $i\neq j$. Regularizing $f$ makes it strongly convex, and
this will be shown to provide robustness to asynchrony below. The
regularized form of $f$ is denoted $f_{A}$, and is defined as
\begin{equation}
f_{A}\left(x\right)\coloneqq f\left(x\right)+\frac{1}{2}x^{T}Ax,
\end{equation}
where $A=\textrm{diag}\left(\alpha_{1}I_{n_{1}},\ldots,\alpha_{N}I_{n_{N}}\right),$
and where $I_{n_{i}}$ is the $n_{i}\times n_{i}$ identity matrix.

In some optimization settings, some decision variables
evolve at drastically different numerical scales \cite{bishop1995neural}. To more
meaningfully evaluate the convergence of agents with respect to one
another, it would be useful to normalize each agent's distance to an optimum
to prevent the error of one agent dominating the convergence
analysis. Allowing heterogeneous normalizations would therefore give a more
useful estimate of the distance to an optimum, and this should be
accounted for by our framework. Moreover, each agent may wish to evaluate the
convergence of its own state using a particular $p$-norm. Therefore, our framework should accommodate
agents measuring the distance to an optimum in different norms. Bearing these
criteria in mind, we now state
the problem that is the focus of the rest of the paper. 
\begin{problem}
For a team of $N$ agents,
\begin{equation}
\underset{x\in X}{\textrm{minimize }}f_{A}\left(x\right)
\end{equation}
while measuring convergence with heterogeneous normalization constants and norms across the
agents. \hfill$\lozenge$
\end{problem}

Section III specifies the structure of the asynchronous communications and computations used to solve Problem~1.

\section{Block-Based Multi-Agent Update Law}
To define the exact update law for each agent's state,
we must describe what information is stored and how agents communicate.
Each agent will store a vector containing its own states
and those of agents it communicates with. Each agent only updates
its own states within the vector it stores onboard. States stored onboard
agent $i$ which correspond to other agents' states are only updated
when those agents send their states to agent $i$. This type of block-based
update can be used to capture, for example, when an agent does not have the
information required to update other agents' states, or when it is desirable
to parallelize updates to reduce each agent's computational burden.

Formally, we will denote agent $i$'s full vector of states by $x^{i}$. Agent
$i$'s own states in this vector are then denoted by $x_{i}^{i}$. The current
values stored onboard agent $i$ for agent $j$ are denoted by $x_{j}^{i}$.
At timestep $k$, agent $i$'s full state vector is denoted $x^{i}\left({k}\right)$,
with its own states denoted $x_{i}^{i}\left({k}\right)$ and those of agent
$j$ denoted $x_{j}^{i}\left({k}\right)$. At any single timestep, agent $i$
may or may not update its states due to asynchrony in agents' computations, and the times of these
updates must be accounted for. We define the set $K^{i}$ to be the
collection of time indices $k$ at which agent $i$ updates $x_{i}^{i}$; agent
$i$ does not compute an update for time
indices $k\notin K^{i}$. Using this notation, agent $i$'s update
law can be written as
\begin{equation}
x_{i}^{i}\left(k+1\right)=\left\{ \begin{array}{cc}
x_{i}^{i}\left(k\right)-\gamma\nabla_{i}f_{A}\left(x^{i}\left(k\right)\right) & k\in K^{i}\\
x_{i}^{i}\left(k\right) & k\notin K^{i}
\end{array}\right.,
\end{equation}
where agent $i$ uses stepsize $\gamma >0$, which will be bounded below.
Here $\nabla_{i}f_{A}\coloneqq\frac{\partial f_{A}}{\partial x_{i}}$
is the gradient of the regularized cost function with respect to $x_{i}$. The
significance of agent $i$'s choice of regularization parameter can
be seen by expanding $\nabla_{i}f_{A}\left(x^{i}\left(k\right)\right)$
as $\nabla_{i}f_{A}\left(x^{i}\left(k\right)\right)=\nabla_{i}f\left(x^{i}\left(k\right)\right)+\alpha_{i}x_{i}^{i}\left(k\right)$,
where $\alpha_{i}>0$ is set by agent~$i$ alone.


In order to account for communication delays we use $\tau_{j}^{i}\left(k\right)$
to denote the time at which the value of
$x_{j}^{i}\left(k\right)$ was originally computed by agent $j$. For example,
if agent $j$ computes a state update at time $k_{a}$ and immediately
transmits it to agent $i$, then agent $i$ may receive this state update at time
$k_{b}>k_{a}$ due to communication delays. Then $\tau_{j}^{i}$ is defined so
that $\tau_{j}^{i}\left(k_{b}\right)=k_{a}$, the time at which agent $j$
originally computed the update just received by agent $i$.
Concerning $K^{i}$ and $\tau_{j}^{i}\left(k\right)$, we have the following assumption.
\begin{assumption}
For all $i\in\left[N\right]$, the set $K^{i}$ is infinite. Moreover, for all $i\in\left[N\right]$ and $j\in\left[N\right]\backslash\left\{i\right\}$, if $\left\{k_{d}\right\}_{d\in\mathbb{N}}$
is a sequence in $K^{i}$ tending to infinity,
then
\begin{equation}
\lim_{d\rightarrow\infty} \tau_{j}^{i}\left(k_{d}\right)=\infty.
\end{equation} \hfill$\triangle$
\end{assumption}
Assumption~3 is quite mild in that it simply requires that no agent ever
permanently stop updating and sharing information. 
For $i\neq j$, the sets $K^{i}$ and $K^{j}$ need not have any relationship because agents' updates are asynchronous. The entire update law for all agents can then be written as follows.
\begin{talgorithm}
For all $i\in\left[N\right]$ and $j\in\left[N\right]\backslash\left\{i\right\}$, execute
\begin{equation}
\begin{aligned}x_{i}^{i}\!\left(k\!+\!1\right) & \!=\!\left\{ \!\!\!\begin{array}{cc}
x_{i}^{i}\left(k\right)-\gamma\nabla_{i}f_{A}\left(x^{i}\left(k\right)\right) & k\in K^{i}\\
x_{i}^{i}\left(k\right) & k\notin K^{i}
\end{array}\right.\\
x_{j}^{i}\!\left(k\!+\!1\right) & \! =\!\left\{ \!\!\!\begin{array}{cc}
x_{j}^{j}\left(\tau_{j}^{i}\left(k+1\right)\right) & \!\!\!\!\textrm{i receives j's state at k+1}\\
x_{j}^{i}\left(k\right) & \textrm{otherwise}
\end{array}\right.\!\!\!\!\!.
\end{aligned}
\end{equation} \hfill$\diamond$
\end{talgorithm}

In Algorithm 1 we see that $x_{j}^{i}$ changes only when agent $i$ receives a
transmission from agent $j$; otherwise it remains constant. Agent $i$ can
therefore reuse old values of agents $j$'s state many times and can reuse
different agents' states different numbers of times. Showing convergence of this
update law must take these delays into account, and that is the subject of the
next section.

\section{Convergence of Asynchronous Optimization}
In this section we prove the convergence of the multi-agent block update law in
Algorithm~1. We first define the block-maximum norm used to
measure convergence and then define a collection of nested sets that will be used
to show asymptotic convergence of all agents adapted from the approach in~\cite{Hale2017}. Then a convergence rate is developed using parameters from
these sets. 
\subsection{Block-Maximum Norms}
We begin by analyzing the convergence of the optimization algorithm
using block maximum norms similar to those defined in \cite{Bertsekas1989}, \cite{bertsekas1989parallell},
and \cite{Hale2017}, and we do so to accommodate the need for heterogeneus normalizations and norms in Problem~1. Due to asynchrony
in agents' communications, we will generally have $x^{i}\left(k\right)\neq x^{j}\left(k\right)$
for all agents $i$ and $j$ and all timesteps $k$. We will refer
to $x_{i}^{i}$ as the $i^{th}$ block of $x^{i}$ and $x_{j}^{i}$
as the $j^{th}$ block of $x^{i}$. With these blocks defined we next define the block-maximum norm that will be used to measure convergence below.
\begin{definition}
Let $x\in\mathbb{R}^{n}$ consist of $N$ blocks, with $x_{i}\in\mathbb{R}^{n_{i}}$
being the $i^{th}$ block. The $i^{th}$ block is weighted by some normalization constant
$w_{i}\geq1$ and is measured in the $p_{i}$-norm for some $p_{i}\in\left[1,\infty\right]$.
The norm of the full vector $x$ is defined as the maximum norm of any single block, i.e.,
\begin{equation}
\left\Vert x\right\Vert _{\max}\coloneqq\underset{i\in\left[N\right]}{\max}\frac{\left\Vert x_{i}\right\Vert _{p_{i}}}{w_{i}}.
\end{equation}\hfill$\triangle$
\end{definition}

The following lemma allows us to upper-bound the induced block-maximum
matrix norm by the Euclidian matrix norm, which will be used below in our
convergence analysis. In this lemma, we use the notion of a block of an
$n\times n$ matrix. Given a matrix $B\in\mathbb{R}^{n\times n}$, where
$n=\sum_{i=1}^{N}n_{i}$, the $i^{th}$ block of $B$, denoted $B^{\left[i\right]}$,
is the $n_{i}\times n$ matrix formed by rows of $B$ with indices $\sum_{k=1}^{i-1}n_{k}+1$ through $\sum_{k=1}^{i}n_{k}$. We then have the following result.
\begin{lemma}
Let $p_{\min}\coloneqq\min_{i\in\left[{N}\right]} p_{i}$ and let $w_{\min}=\min_{i\in\left[{N}\right]} w_{i}$. Then for all $B\in\mathbb{R}^{n\times n}$,
\begin{equation}
\left\Vert B\right\Vert _{\max}\leq\left\{ \begin{array}{cc}
n^{\left(p_{\min}^{-1}-\frac{1}{2}\right)}w_{\min}^{-1}\left\Vert B\right\Vert _{2} & p_{\min}<2\\
\frac{1}{w_{\min}}\left\Vert B\right\Vert _{2} & p_{\min}\geq2
\end{array}\right..
\end{equation}
\end{lemma}
\begin{proof}
For $B^{\left[i\right]}$ the $i^{th}$ block of $B$ and any $x\in\mathbb{R}^{n}$,
by definition we have
\begin{equation}
\begin{aligned}\frac{\left\Vert B^{\left[i\right]}x\right\Vert _{p_{i}}}{w_{i}} & =\frac{1}{w_{i}}\left(\sum_{k=1}^{n_{i}}\left|\sum_{j=1}^{n}B_{k,j}^{\left[i\right]}x_{j}\right|^{p_{i}}\right)^{\frac{1}{p_{i}}}.\end{aligned}
\label{eq:BlockNorm}
\end{equation}
From the definition of a $p$-norm, the right side of Equation~\eqref{eq:BlockNorm}
will always be non-negative. Thus summing the right-hand side over every block
results in
\begin{equation}
\begin{aligned}\frac{\left\Vert B^{\left[i\right]}x\right\Vert _{p_{i}}}{w_{i}} & \leq\sum_{i=1}^{N}\left(\frac{1}{w_{i}}\left(\sum_{k=1}^{n_{i}}\left|\sum_{j=1}^{n}B_{k,j}^{\left[i\right]}x_{j}\right|^{p_{i}}\right)^{\frac{1}{p_{i}}}\right)\end{aligned}.
\end{equation}
Next,
recalling that $\left\Vert{x}\right\Vert_{q}\leq\left\Vert{x}\right\Vert_{r}$
for all vectors $x\in\mathbb R^{n}$ and all $q\geq r>0$, we find that
\begin{equation}
\begin{aligned}\frac{\left\Vert B^{\left[i\right]}x\right\Vert _{p_{i}}}{w_{i}} & \leq\frac{1}{w_{min}}\sum_{i=1}^{N}\left(\sum_{k=1}^{n_{i}}\left|\sum_{j=1}^{n}B_{k,j}^{\left[i\right]}x_{j}\right|^{p_{i}}\right)^{\frac{1}{p_{i}}}\\
 & \leq\frac{1}{w_{\min}}\sum_{i=1}^{N}\left(\sum_{k=1}^{n_{i}}\left|\sum_{j=1}^{n}B_{k,j}^{\left[i\right]}x_{j}\right|^{p_{\min}}\right)^{\frac{1}{p_{\min}}}.
\end{aligned}
\end{equation}
This then allows us to express the sum over all rows of $B$ via
\begin{equation}
\frac{\left\Vert B^{\left[i\right]}x\right\Vert _{p_{i}}}{w_{i}}\leq\frac{1}{w_{min}}\left(\sum_{l=1}^{n}\left|\sum_{j=1}^{n}B_{l,j}x_{j}\right|^{p_{\min}}\right)^{\frac{1}{p_{\min}}}.
\end{equation}
If $p_{min}\geq2$, then $\left\Vert B^{\left[i\right]}x\right\Vert _{p_{i}}\leq\left\Vert B^{\left[i\right]}x\right\Vert _{2}$ for all $p_i$.
If $p_{min}<2$, we recall that $\left\Vert x\right\Vert _{l}\leq\left\Vert x\right\Vert _{p_{\min}}\leq n^{\left(p_{\min}^{-1}-l^{-1}\right)}\left\Vert x\right\Vert _{l}$,
which follows from H{\"o}lder's inequality for $0<p_{min}<l$, and observe that
$\left\Vert B^{\left[i\right]}x\right\Vert _{p_{i}}\leq\left\Vert B^{\left[i\right]}x\right\Vert _{p_{\min}}\leq n^{\left(p_{\min}^{-1}-\frac{1}{2}\right)}\left\Vert Bx\right\Vert _{2}$.
Combining these inequalities we find that
\begin{equation}
\frac{\left\Vert B^{\left[i\right]}x\right\Vert _{p_{i}}}{w_{i}}\leq\left\{ \begin{array}{cc}
n^{\left(p_{\min}^{-1}-\frac{1}{2}\right)}w_{\min}^{-1}\left\Vert Bx\right\Vert _{2} & p_{\min}<2\\
\frac{1}{w_{\min}}\left\Vert Bx\right\Vert _{2} & p_{\min}\geq2
\end{array}\right.
\end{equation}
for all $i$. Thus the weighted block maximum norm of $Bx$ for any $x\in\mathbb{R}^{n}$
can be bounded as 
\begin{equation}
\begin{aligned}\left\Vert Bx\right\Vert _{\max} & =\underset{i\in\left[N\right]}{\max}\frac{\left\Vert B^{\left[i\right]}x\right\Vert _{p_{i}}}{w_{i}}\\
 & \leq\left\{ \begin{array}{cc}
n^{\left(p_{\min}^{-1}-\frac{1}{2}\right)}w_{\min}^{-1}\left\Vert Bx\right\Vert _{2} & p_{\min}<2\\
\frac{1}{w_{\min}}\left\Vert Bx\right\Vert _{2} & p_{\min}\geq2
\end{array}\right.,
\end{aligned}
\label{eq:BlockNormBound}
\end{equation}
and the lemma follows by taking the supremum over all unit vectors $x$.
\end{proof}
\subsection{Convergence Via Nested Sets}
We now begin the convergence analysis for the block-based update law in Algorithm~1 where agents are asynchronously optimizing.
In order for this system to converge using the communications
described in the previous section, we construct a sequence of sets,
$\left\{ X\left(s\right)\right\} _{s\in\mathbb{N}}$, based on work in \cite{Bertsekas1989} and \cite{bertsekas1989parallell}. Below we use the notation $\hat{x}_{A}\coloneqq\arg\min_{x\in X}f_{A}\left(x\right)$
to specify the minimizer of the regularized cost function $f_{A}$. We state the
conditions imposed upon these sets as an assumption, and this assumption will be
shown below to be satisfied using the heterogeneous regularization applied by $A$.
\begin{assumption}
The sets $\left\{ X\left(s\right)\right\} _{s\in\mathbb{N}}$ satisfy:
\begin{enumerate}
\item $\cdots\subset X\left(s+1\right)\subset X\left(s\right)\subset\cdots\subset X$
\item $\underset{s\rightarrow\infty}{\lim}X\left(s\right)=\left\{ \hat{x}_{A}\right\} $ 
\item $X_{i}\left(s\right)\subset X_{i}$ for all $i\in\left[N\right]$
and $s\in\mathbb{N}$ such that $X\left(s\right)=X_{1}\left(s\right)\times\cdots\times X_{N}\left(s\right)$
\item $\theta_{i}\left(y\right)\in X_{i}\left(s+1\right)$, where $\theta_{i}\left(y\right)\coloneqq y_{i}-\gamma\nabla_{i}f_{A}\left(y\right)$
for all $y\in X\left(s\right)$ and $i\in\left[N\right]$.\hfil$\triangle$
\end{enumerate} 
\end{assumption}

Assumptions 4.1 and 4.2 together show that these sets are nested as
they converge to the minimum $\hat{x}_{A}$. Assumption~4.3 allows for
the blocks to be updated independently by the agents, and Assumption
4.4 ensures that state updates always progress down the chain of nested sets such that only
forward progress toward $\hat{x}_{A}$ is made. It is shown in \cite{Bertsekas1989} and \cite{bertsekas1989parallell} that the existence of such a sequence of sets implies asymptotic convergence of the asynchronous update law in Algorithm~1, and we therefore use this construction to show asymptotic convergence in this paper. Defining the Lipschitz constant of $\nabla_{i}f_{A}$ as
$L_{i}$, we further define $L_{\max}\coloneqq\underset{i\in\left[N\right]}{\max}\ L_{i}$,
 and then define the constant
\begin{equation}
q=\max\left\{ \underset{i\in\left[N\right]}{\max}\left|1-\gamma\alpha_{i}\right|,\underset{i\in\left[N\right]}{\max}\left|1-\gamma L_{i}\right|\right\} .
\end{equation}
Letting $\gamma\in\left(0,\frac{2}{L_{\max}}\right)$ and $\alpha\in\left(0,L_{\max}\right)$,
we find $q\in\left(0,1\right)$; a proof for this can be seen in \cite{polyak1987}.
We then proceed to define $D_{o}$ as
\begin{equation}
D_{o}\coloneqq\underset{i\in\left[N\right]}{\max}\left\Vert x^{i}\left(0\right)-\hat{x}_{A}\right\Vert _{\max},
\end{equation}
which is the worst-performing block onboard any agent with respect to
distance to $\hat{x}_{A}$ at timestep $0$. We then define the sequence of
sets $\left\{ X\left(s\right)\right\} _{s\in\mathbb{N}}$ as
\begin{equation}\label{eq:Xsets}
X\left(s\right)=\left\{ y\in X:\left\Vert y-\hat{x}_{A}\right\Vert _{\max}\leq q^{s}D_{o}\right\},
\end{equation}
and this construction is shown in the following theorem to satisfy Assumption~4, thereby ensuring asymptotic convergence of Algorithm~1.
\begin{theorem}
The collection of sets $\left\{ X\left(s\right)\right\} _{s\in\mathbb{N}}$
as defined in Equation~\eqref{eq:Xsets} satisfies Assumption 4. 
\end{theorem}
\begin{proof}
For Assumption~4.1 we see that
\begin{equation}
X\left(s+1\right)=\left\{ y\in X:\left\Vert y-\hat{x}_{A}\right\Vert _{\max}\leq q^{s+1}D_{o}\right\} .
\end{equation}
Since $q\in\left(0,1\right)$, we have $q^{s+1}<q^{s}$, which results
in $\left\Vert y-\hat{x}_{A}\right\Vert_{\max} \leq q^{s+1}D_{o}<q^{s}D_{o}$.
Then $y\in{X}\left(s+1\right)$ implies $y\in{X}\left(s\right)$ and $X\left(s+1\right)\subset X\left(s\right)\subset X$, as desired.

From Assumption~4.2 we find
\begin{equation}
\begin{aligned}\underset{s\rightarrow\infty}{\lim}X\left(s\right) & =\underset{s\rightarrow\infty}{\lim}\left\{ y\in X:\left\Vert y-\hat{x}_{A}\right\Vert _{\max}\leq q^{s}D_{o}\right\} \\
 & =\left\{ y\in X:\left\Vert y-\hat{x}_{A}\right\Vert _{\max}\leq0\right\} \\
 & =\left\{ \hat{x}_{A}\right\} ,
\end{aligned}
\end{equation}
and Assumption 4.2 is therefore satisfied. The structure of the weighted
block-maximum norm then allows us to see that
$\left\Vert y-\hat{x}_{A}\right\Vert _{\max}\leq q^{s}D_{o}$ if and 
only if $\frac{1}{w_{i}}\left\Vert y_{i}-\hat{x}_{A,i}\right\Vert _{p_{i}}\leq q^{s}D_{o}$
for all $i\in\left[N\right]$. It then follows that
\begin{equation}
X_{i}\left(s\right)=\left\{ y_{i}\in X_{i}:\frac{1}{w_{i}}\left\Vert y_{i}-\hat{x}_{A,i}\right\Vert _{p_{i}}\leq q^{s}D_{o}\right\} ,
\end{equation}
which shows $X\left(s\right)=X_{1}\left(s\right)\times\cdots\times X_{N}\left(s\right)$, thus satisfying Assumption 4.3.

In order to show Assumption 4.4 is satisfied we recall the following exact expansion of $\nabla f_{A}$:

\begin{equation}
\begin{aligned} \nabla f_{A}\left(y\right)-\nabla f_{A}\left(\hat{x}_{A}\right)
& \ =\int_{0}^{1}\nabla^{2}f_{A}\left(\hat{x}_{A}+\tau\left(y-\hat{x}_{A}\right)\right)\left(y-\hat{x}_{A}\right)d\tau\\
& \ =\left(\int_{0}^{1}\nabla^{2}f_{A}\left(\hat{x}_{A}+\tau\left(y-\hat{x}_{A}\right)\right)d\tau\right)\cdot\left(y-\hat{x}_{A}\right)\\
 & \ \eqqcolon H\left(y\right)\left(y-\hat{x}_{A}\right),
\end{aligned}
\label{eq:heshian}
\end{equation}
where we have defined
\begin{equation}
H\left(y\right)=\int_{0}^{1}\nabla^{2}f_{A}\left(\hat{x}_{A}+\tau\left(y-\hat{x}_{A}\right)\right)d\tau.
\end{equation}

We then see that for $y\in X\left(s\right)$,

\begin{equation}
\begin{aligned} \frac{\left\Vert \theta_{i}\left(y\right)-\hat{x}_{A,i}\right\Vert _{p_{i}}}{w_{i}}
& =\frac{1}{w_{i}}\left\Vert y_{i}-\gamma\nabla_{i}f_{A}\left(y\right)-\hat{x}_{A,i}+\gamma\nabla_{i}f_{A}\left(\hat{x}_{A}\right)\right\Vert _{p_{i}}\\
& \leq\underset{i\in\left[N\right]}{\max}\frac{1}{w_{i}}\left\Vert y_{i}-\gamma\nabla_{i}f_{A}\left(y\right)-\hat{x}_{A,i}+\gamma\nabla_{i}f_{A}\left(\hat{x}_{A}\right)\right\Vert _{p_{i}}\\
 & =\left\Vert y-\hat{x}_{A}-\gamma\nabla f_{A}\left(y\right)+\gamma\nabla f_{A}\left(\hat{x}_{A}\right)\right\Vert _{\max}\\
 & =\left\Vert y-\hat{x}_{A}-\gamma\left(\nabla f_{A}\left(y\right)-\nabla f_{A}\left(\hat{x}_{A}\right)\right)\right\Vert _{\max}\\
 & =\left\Vert y-\hat{x}_{A}-\gamma H\left(y\right)\left(y-\hat{x}_{A}\right)\right\Vert _{\max}\\
 & \leq\left\Vert I-\gamma H\left(y\right)\right\Vert _{\max}\left\Vert {y-\hat{x}_{A}}\right\Vert _{\max}\\
 & \leq\left\{ \begin{array}{cc}
\frac{n^{\left(p_{\min}-\frac{1}{2}\right)}}{w_{\min}}\left\Vert I\!\!-\!\gamma H\!\left(y\right)\right\Vert _{2}\left\Vert {y-\hat{x}_{A}}\right\Vert _{\max} & p_{\min}<2\\[5pt]
\frac{1}{w_{\min}}\left\Vert I-\gamma H\left(y\right)\right\Vert _{2}\left\Vert{y-\hat{x}_{A}}\right\Vert _{\max} & p_{\min}\geq2
\end{array}\right.,
\end{aligned}
\end{equation}
where we have used Equation~\eqref{eq:heshian} in the fourth equality and Lemma~1 in the third inequality. We then define the vector
$\nabla f_{A}=\left(\nabla_{1}f_{A},\ldots,\nabla_{N}f_{A}\right)^{T}$
which has a Lipschitz constant of $M=\sqrt{\sum_{i=1}^{N}L_{i}^{2}}$.
It then follows from the definition of $f_{A}$ that $A\preceq H\left(\cdot\right)\preceq MI$,
which implies that the eigenvalues of $H\left(\cdot\right)$ are bounded
below by the smallest diagonal entry of $A$ and above by $M$. Since $H\left(y\right)$
is a symmetric matrix it follows that

\begin{equation}
\begin{aligned} \left\Vert I-\gamma H\left(y\right)\right\Vert _{2}
 &  =\max\left\{ \left|\lambda_{\min}\left(I-\gamma H\left(y\right)\right)\right|,\left|\lambda_{\max}\left(I-\gamma H\left(y\right)\right)\right|\right\} \\
 &  =\max\left\{ \underset{i\in\left[N\right]}{\max}\left|1-\gamma\alpha_{i}\right|,\underset{i\in\left[N\right]}{\max}\left|1-\gamma L_{i}\right|\right\} \\
 &  =q,
\end{aligned}
\end{equation}
where $\lambda_{\min}\left(\cdot\right)$ and $\lambda_{\max}\left(\cdot\right)$
are the minimum and maximum eigenvalues of a matrix, respectively.
Using the hypothesis that $y\in{X}\left(s\right)$, we find
\begin{equation}
\begin{aligned} \frac{\left\Vert \theta_{i}\left(y\right)-\hat{x}_{A,i}\right\Vert _{p_{i}}}{w_{i}}
& \leq\left\{ \begin{array}{cc}
n^{\left(p_{\min}^{-1}-\frac{1}{2}\right)}w_{\min}^{-1}q\left\Vert y-\hat{x}_{A}\right\Vert _{\max} & p_{\min}<2\\
\frac{1}{w_{\min}}q\left\Vert y-\hat{x}_{A}\right\Vert _{\max} & p_{\min}\geq2
\end{array}\right.\\
 & \leq\left\{ \begin{array}{cc}
n^{\left(p_{\min}^{-1}-\frac{1}{2}\right)}w_{\min}^{-1}q^{s+1}D_{o} & p_{\min}<2\\
\frac{1}{w_{\min}}q^{s+1}D_{o} & p_{\min}\geq2
\end{array}\right.\\
 & \leq\left\{ \begin{array}{c}
q^{s+1}D_{o}\\
q^{s+1}D_{o}
\end{array}\right.,
\end{aligned}
\end{equation}
where the bottom case follows from $w_{\min}\geq {1}$ and the top case follows
from $w_{\min}\geq {1}$ and $p_{\min}^{-1}-\frac{1}{2}<1$. Then $\theta_{i}\left(y\right)\in X_{i}\left(s+1\right)$ and Assumption 4.4 is satisfied.
\end{proof}
As noted above, the fact that the construction in Equation~\eqref{eq:Xsets} satisfies Assumption 4 implies asymptotic convergence of Algorithm~1 for all $i\in\left[N\right]$ from \cite{Bertsekas1989} and \cite{bertsekas1989parallell}. With this in mind, we next derive a rate of convergence for Algorithm~1.

\subsection{Convergence Rate}
The structure of the sets $\left\{ X\left(s\right)\right\} _{s\in\mathbb{N}}$
allows us to determine a convergence rate. However, to do so we must
first define the notion of a \emph{communication cycle}. Starting at time $k=0$, one cycle occurs
when all agents have calculated a state update and this updated state
has been sent to and received by each other agent. It
is only then that each agents' copy of the ensemble state is moved
from $X\left(0\right)$ to $X\left(1\right)$. Once another cycle
is completed the ensemble state is moved from $X\left(1\right)$ to
$X\left(2\right)$. This process repeats indefinitely, and coupled
with Assumption 4, means the convergence rate is geometric in the
number of cycles completed, which we show now.
\begin{theorem}
Let Assumptions 1-4 hold and let $\gamma\in\left(0,\frac{2}{L_{\max}}\right)$.
At time $k$, if $c\left(k\right)$ cycles have been completed, then
\begin{equation}
\left\Vert x^{i}\left(k\right)-\hat{x}_{A}\right\Vert _{\max}\leq q^{c\left(k\right)}D_{o}
\end{equation}
for all $i\in\left[N\right]$.
\end{theorem}
\begin{proof}
From the definition of $D_{o}$, for all $i\in\left[N\right]$ we
have $x^{i}\left(0\right)\in X\left(0\right)$. If agent $i$ computes
a state update, then $\theta_{i}\left(x^{i}\left(0\right)\right)\in X_{i}\left(1\right)$
and after one cycle is completed, say at time $k$, we have $x^{i}\left(k\right)\in X\left(1\right)$ for all $i$.
Iterating this process, after $c\left(\bar{k}\right)$ cycles have been completed by some time $\bar{k}$, $x^{i}\left(\bar{k}\right)\in X\left(c\left(\bar{k}\right)\right)$. The result follows by expanding the definition of $\{X\left(s\right)\}_{s\in\mathbb{N}}$.
\end{proof}

Theorem~3 can be used by a network operator to bound agents' convergence by simply observing them and without specifying when or how often agents should generate or share information. Having shown convergence of Algorithm~1, we next demonstrate its performance in practice.

\section{Simulation}
In this section we present a problem to be solved using Algorithm~1. The simulation uses a
network consisting of 8 nodes and 9 edges, where we define the set $\varepsilon\coloneqq\left[9\right]$ as the set of indices of the edges.
There are $N=8$ agents that are users of this network and they are each tasked with routing a flow between two nodes. The network itself is shown in Figure~\ref{fig:network}; we emphasize that the nodes in the network are not the agents themselves, but instead are simply source/destination pairs for users to route flows between.
The starting and ending nodes
as well as the edges traversed for each agents' flow are listed in
Table~\ref{tab:Edges-Traversed-by}.
\begin{figure}[H]
\centering
\includegraphics[width=2.6in]{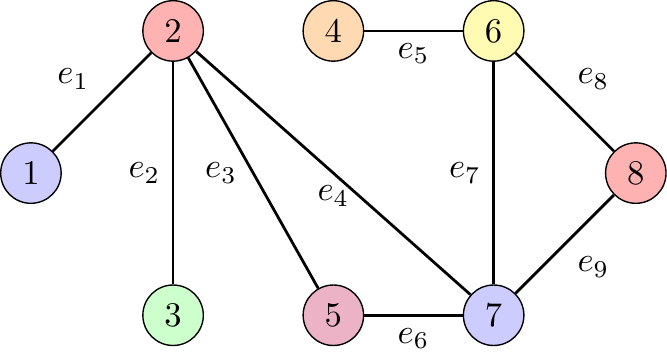}
\caption{The network through which eight agents must route a flow between two nodes}
\label{fig:network}
\end{figure}
\begin{table}[H]
\centering{}%
\begin{tabular}{|c|c|c|}
\hline 
Agent Number & Start Node$\rightarrow$End Node & Edges Traversed\tabularnewline
\hline 
\hline 
1 & $1\rightarrow7$ & $e_{1},e_{3},e_{6}$\tabularnewline
\hline 
2 & $2\rightarrow8$ & $e_{4},e_{7},e_{8}$\tabularnewline
\hline 
3 & $3\rightarrow4$ & $e_{2},e_{4},e_{7},e_{5}$\tabularnewline
\hline 
4 & $5\rightarrow6$ & $e_{3},e_{4},e_{7}$\tabularnewline
\hline 
5 & $1\rightarrow4$ & $e_{1},e_{3},e_{6},e_{7},e_{5}$\tabularnewline
\hline 
6 & $3\rightarrow8$ & $e_{2},e_{4},e_{9}$\tabularnewline
\hline 
7 & $4\rightarrow5$ & $e_{5},e_{8},e_{9},e_{6}$\tabularnewline
\hline 
8 & $6\rightarrow2$ & $e_{7},e_{4}$\tabularnewline
\hline 
\end{tabular}\caption{Edges traversed by each agent's flow
\label{tab:Edges-Traversed-by}}
\end{table}

The cost function of agent $i$ is $f_{i}\left(x_{i}\right)=-100\log\left(1+x_{i}\right)$,
and the coupling cost is $c\left(x\right)=\frac{1}{20}x^{T}C^{T}Cx$,
where the network connection matrix is defined as
\begin{equation}
C_{k,i}=\left\{ \begin{array}{cc}
1 & \textrm{if flow }i\textrm{ traverses edge }k\\
0 & \textrm{otherwise}
\end{array}\right..
\end{equation}

This problem was then implemented such that agent $i$ had its own
regularization parameter $\alpha_{i}>0$, normalization constant $w_{i}\geq 1$, and $p_{i}$ norm with $p_{i}\in\left[ 1,\infty\right]$. In particular, these parameters were chosen using $w = [12,8,6,7,6,10,9,10]$ and
$p = [\infty,20,3,90,6,12,2,9]$, where $w_{i}$ is the $i^{th}$ element in $w$ and $p_{i}$ is defined analogously. All agents' behaviors were randomized to give each agent a $10\%$ chance of computing an update at any timestep and to give each pair of agents a $10\%$ chance of communicating at each timestep. Three total simulation runs were executed using the three different choices of $A$ listed to demonstrate its effects upon convergence, with

\begin{equation}
\begin{aligned}
A_1 & = \textnormal{diag}[3\!\!\times \!\!10^{-4},1\!\!\times \!\!10^{-4},9\!\!\times \!\!10^{-4},2\!\!\times \!\!10^{-4},0.001,0.001,5\!\! \times \!\! 10^{-4},4\!\! \times \!\! 10^{-4}] \\
A_2 & = \textnormal{diag}[0.01,0.01,0.003,0.005,0.002,0.01,0.005,0.002] \\
A_3 & = \textnormal{diag}[0.08,0.1,0.1,0.09,0.009,0.1,0.08,0.04].
\end{aligned}
\end{equation}

A plot of error versus iteration count for a run with $A_1$ is shown in Figure~\ref{fig:error0.001}, which shows that the regularization provided by $A_1$ can provide robustness to asynchrony without significantly impacting the final point obtained by Algorithm~1. In addition, close convergence to a minimizer is attained in a reasonable number of iterations, even when agents infrequently generate and share information.

\begin{figure}[H]
\centering
\includegraphics[width=3.3in]{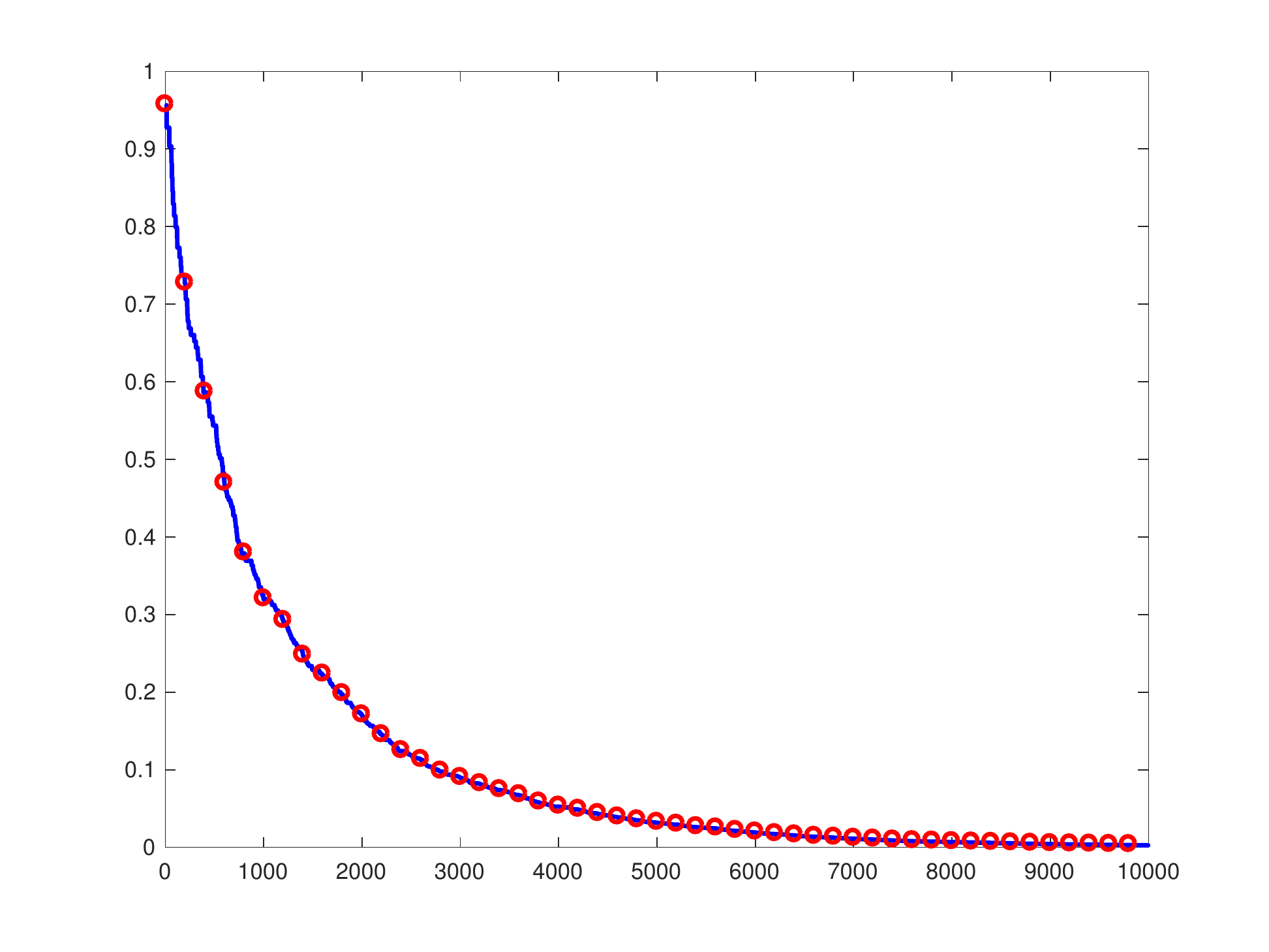}
\caption{Regularized and unregularized error for agent 1 where $\left\Vert{A_{1}}\right\Vert =0.001$. Here, the regularized error is shown as a line and the unregularized error is shown by the circles. As expected, both errors converge to small final values, indicating close convergence to both $\hat x$ and $\hat{x}_{A}$ when $\left\Vert{A}\right\Vert$ is small.}
\label{fig:error0.001}
\end{figure}

To demonstrate the impact of larger regularizations, a simulation was run with $A_{2}$, and an error plot for this run is shown in Figure~\ref{fig:error0.01}.
\begin{figure}[H]
\centering
\includegraphics[width=3.3in]{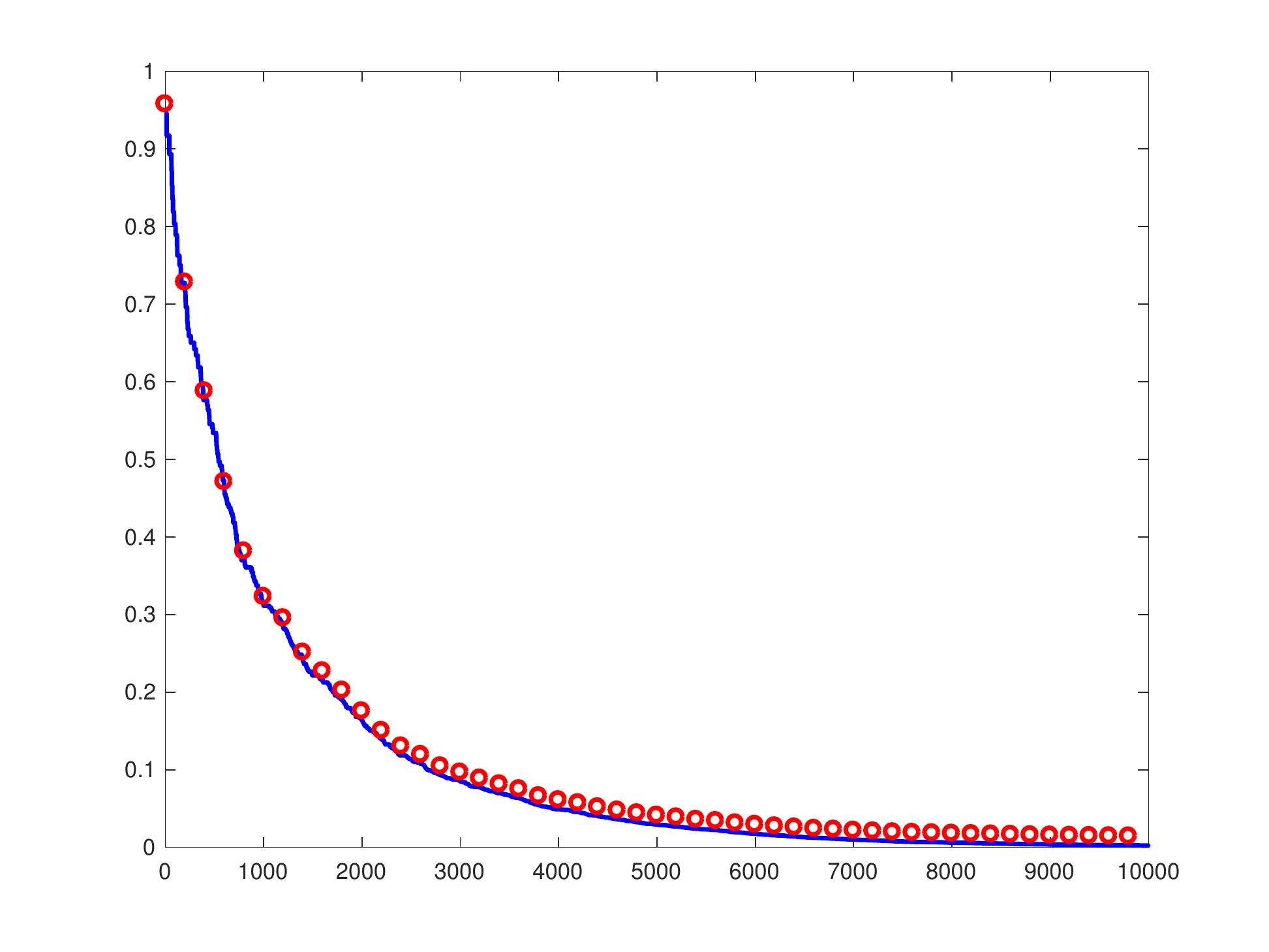}
\caption{Regularized and unregularized error for agent 1 where $\left\Vert{A_{2}}\right\Vert =0.01$. The regularized error is shown as a line and the unregularized error is shown by the circles. Because $\left\Vert{A}\right\Vert$ is larger, the agents converge to a minimum faster, though there is a larger discrepancy between $\hat x$ and $\hat{x}_{A}$, as evidenced by the asymptotic disagreement between the two curves shown here.}
\label{fig:error0.01}
\end{figure}

To further illustrate the effects of regularizing, a third and final simulation was run with $A_3$, and a plot of error in this case is shown in Figure~\ref{fig:error0.1}.
\begin{figure}[H]
\centering
\includegraphics[width=3.3in]{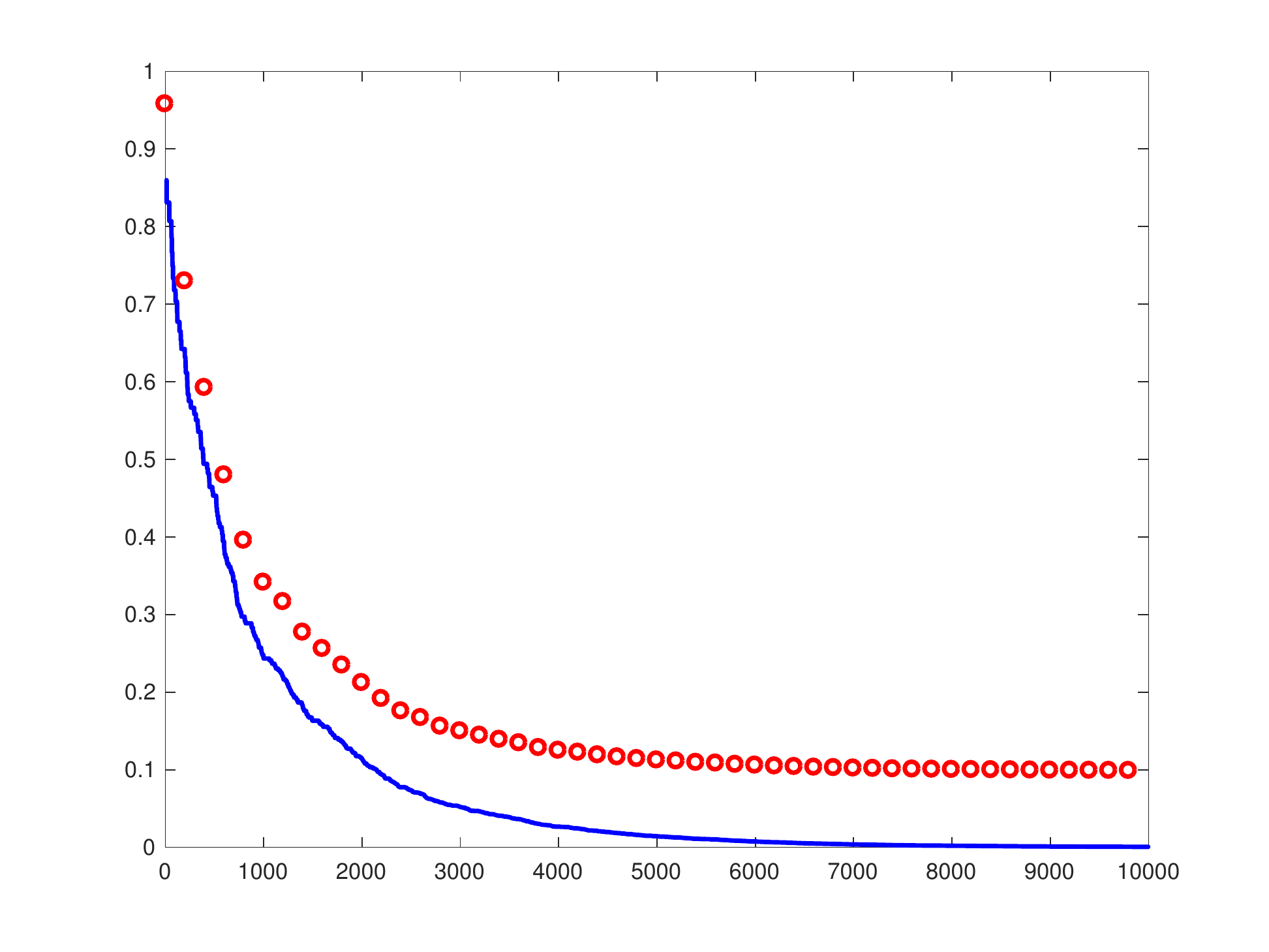}
\caption{Regularized and unregularized error for agent 1 where $\left\Vert{A_{3}}\right\Vert =0.1$. The regularized error is shown as a line and the unregularized error is shown by the circles. As expected, this run converges faster (because its value of $q$ smaller), but with the largest error in the final solution obtained, indicating that a significant acceleration in convergence comes in exchange for a less accurate solution.}
\label{fig:error0.1}
\end{figure}

To enable numerical comparisons of these convergence results, final error values for all three runs are shown in Table~\ref{tab:Errors}, where we see that larger values of $\Vert{A}\Vert$ do indeed lead to larger  errors.
\begin{table}[H]
\centering{}%
\begin{tabular}{|c|c|c|}
\hline 
$\left\Vert{A}\right\Vert$ & Final Regularized Error & Final Unregularized Error \tabularnewline
\hline 
\hline 
0.001 & $2.2575\times 10^{-8}$ & $2.9558\times 10^{-4}$\tabularnewline
\hline 
0.01 & $2.1837\times 10^{-8}$ & $8.4922\times 10^{-4}$\tabularnewline
\hline 
0.1 & $7.9827\times 10^{-10}$ & $0.0848$\tabularnewline
\hline 
\end{tabular}\caption{Errors for agent 1\label{tab:Errors}}
\end{table}

\section{Conclusion}
This work presented an asynchronous optimization framework which allows for arbitrarily delayed communications and computations. Future extensions to this work include incorporating constraints in order to accommodate broader classes of problems \cite{Hale2014}, and using time-varying regularizations to always reach exact solutions. Future applications include use in robotic swarms where communications are unreliable and asynchrony is unavoidable.
 
\bibliographystyle{IEEEtran}{}
\bibliography{CDC}

\begin{thebibliography}{10}
\providecommand{\url}[1]{#1}
\csname url@samestyle\endcsname
\providecommand{\newblock}{\relax}
\providecommand{\bibinfo}[2]{#2}
\providecommand{\BIBentrySTDinterwordspacing}{\spaceskip=0pt\relax}
\providecommand{\BIBentryALTinterwordstretchfactor}{4}
\providecommand{\BIBentryALTinterwordspacing}{\spaceskip=\fontdimen2\font plus
\BIBentryALTinterwordstretchfactor\fontdimen3\font minus
  \fontdimen4\font\relax}
\providecommand{\BIBforeignlanguage}[2]{{%
\expandafter\ifx\csname l@#1\endcsname\relax
\typeout{** WARNING: IEEEtran.bst: No hyphenation pattern has been}%
\typeout{** loaded for the language `#1'. Using the pattern for}%
\typeout{** the default language instead.}%
\else
\language=\csname l@#1\endcsname
\fi
#2}}
\providecommand{\BIBdecl}{\relax}
\BIBdecl

\bibitem{Khan2009}
M.~Khan, G.~Pandurangan, and V.~S.~A. Kumar, ``Distributed algorithms for
  constructing approximate minimum spanning trees in wireless sensor
  networks,'' \emph{IEEE Transactions on Parallel and Distributed Systems},
  vol.~20, no.~1, pp. 124--139, Jan 2009.

\bibitem{Cortes2004}
J.~Cortes, S.~Martinez, T.~Karatas, and F.~Bullo, ``Coverage control for mobile
  sensing networks,'' \emph{IEEE Transactions on Robotics and Automation},
  vol.~20, no.~2, pp. 243--255, April 2004.

\bibitem{Rabbat2004}
M.~Rabbat and R.~Nowak, ``Distributed optimization in sensor networks,'' in
  \emph{Proceedings of the 3rd International Symposium on Information
  Processing in Sensor Networks}, ser. IPSN '04.\hskip 1em plus 0.5em minus
  0.4em\relax New York, NY, USA: ACM, 2004, pp. 20--27.

\bibitem{Mitra1994}
D.~Mitra, \emph{An Asynchronous Distributed Algorithm for Power Control in
  Cellular Radio Systems}.\hskip 1em plus 0.5em minus 0.4em\relax Boston, MA:
  Springer US, 1994, pp. 177--186.

\bibitem{Chiang2007}
M.~Chiang, S.~H. Low, A.~R. Calderbank, and J.~C. Doyle, ``Layering as
  optimization decomposition: A mathematical theory of network architectures,''
  \emph{Proceedings of the IEEE}, vol.~95, no.~1, pp. 255--312, Jan 2007.

\bibitem{Soltero2014}
D.~E. Soltero, M.~Schwager, and D.~Rus, ``Decentralized path planning for
  coverage tasks using gradient descent adaptive control,'' \emph{The
  International Journal of Robotics Research}, vol.~33, no.~3, pp. 401--425,
  2014.

\bibitem{Caron2010}
S.~Caron and G.~Kesidis, ``Incentive-based energy consumption scheduling
  algorithms for the smart grid,'' in \emph{2010 First IEEE International
  Conference on Smart Grid Communications}, Oct 2010, pp. 391--396.

\bibitem{Chen2012}
A.~I. Chen and A.~Ozdaglar, ``A fast distributed proximal-gradient method,'' in
  \emph{Communication, Control, and Computing (Allerton), 2012 50th Annual
  Allerton Conference on}.\hskip 1em plus 0.5em minus 0.4em\relax IEEE, 2012.

\bibitem{Jadbabaie2003}
A.~Jadbabaie, J.~Lin, and A.~S. Morse, ``Coordination of groups of mobile
  autonomous agents using nearest neighbor rules,'' \emph{IEEE Transactions on
  Automatic Control}, vol.~48, no.~6, June 2003.

\bibitem{Moreau2005}
L.~Moreau, ``Stability of multiagent systems with time-dependent communication
  links,'' \emph{IEEE Transactions on Automatic Control}, vol.~50, no.~2, pp.
  169--182, Feb 2005.

\bibitem{Nedic2007}
A.~Nedi{\'{c}} and A.~Ozdaglar, ``\BIBforeignlanguage{English (US)}{On the rate
  of convergence of distributed subgradient methods for multi-agent
  optimization},'' in \emph{\BIBforeignlanguage{English (US)}{Proceedings of
  the 46th IEEE Conference on Decision and Control 2007, CDC}}, 2007, pp.
  4711--4716.

\bibitem{Nedic2009}
A.~Nedi{\'c} and A.~Ozdaglar, ``Distributed subgradient methods for multi-agent
  optimization,'' \emph{IEEE Transactions on Automatic Control}, vol.~54,
  no.~1, pp. 48--61, Jan 2009.

\bibitem{Nedic2010}
A.~Nedi{\'{c}}, A.~Ozdaglar, and P.~A. Parrilo, ``Constrained consensus and
  optimization in multi-agent networks,'' \emph{IEEE Transactions on Automatic
  Control}, vol.~55, no.~4, pp. 922--938, April 2010.

\bibitem{Olshevsky2009}
A.~Olshevsky and J.~N. Tsitsiklis, ``Convergence speed in distributed consensus
  and averaging,'' \emph{SIAM J. Control Optim.}, vol.~48, no.~1, pp. 33--55,
  Feb. 2009.

\bibitem{Ren2005}
W.~Ren and R.~W. Beard, ``Consensus seeking in multiagent systems under
  dynamically changing interaction topologies,'' \emph{IEEE Transactions on
  Automatic Control}, vol.~50, no.~5, pp. 655--661, May 2005.

\bibitem{Touri2009}
B.~Touri and A.~Nedi{\'{c}}, ``Distributed consensus over network with noisy
  links,'' in \emph{2009 12th International Conference on Information Fusion},
  July 2009, pp. 146--154.

\bibitem{Bertsekas1989}
D.~P. Bertsekas and J.~N. Tsitsiklis, ``Convergence rate and termination of
  asynchronous iterative algorithms,'' in \emph{Proceedings of the 3rd
  International Conference on Supercomputing}, ser. ICS '89.\hskip 1em plus
  0.5em minus 0.4em\relax New York, NY, USA: ACM, 1989, pp. 461--470.

\bibitem{bertsekas1989parallell}
D.~P. Bertsekas and J.~Tsitsiklis, ``Parallell and distributed computation,''
  \emph{Upper Saddle River}, 1989.

\bibitem{Hale2017}
M.~T. Hale, A.~Nedi{\'{c}}, and M.~Egerstedt, ``Asynchronous multiagent
  primal-dual optimization,'' \emph{IEEE Transactions on Automatic Control},
  vol.~62, no.~9, pp. 4421--4435, Sept 2017.

\bibitem{bishop1995neural}
C.~Bishop, \emph{Neural Networks for Pattern Recognition}, ser. Advanced Texts
  in Econometrics.\hskip 1em plus 0.5em minus 0.4em\relax Clarendon Press,
  1995.

\bibitem{polyak1987}
B.~T. Polyak, ``Introduction to optimization. translations series in
  mathematics and engineering,'' \emph{Optimization Software}, 1987.

\bibitem{Hale2014}
M.~Hale and Y.~Wardi, ``Mode scheduling under dwell time constraints in
  switched-mode systems,'' in \emph{2014 American Control Conference}, June
  2014, pp. 3954--3959.

\end{thebibliography}

\end{document}